\documentclass[12pt]{article}
\usepackage[symbol]{footmisc}
\usepackage[misc]{ifsym}
\usepackage{mathtools}
\usepackage{amsmath}
\usepackage{amssymb}
\usepackage{amsthm}
\usepackage{mathrsfs}
\usepackage{hyperref}
\usepackage{graphicx}
\usepackage{caption}
\usepackage{subcaption}
\usepackage{xcolor}
\usepackage{listings}
 \usepackage[english]{babel}
\usepackage{bm}
\usepackage{upgreek}
\usepackage{bbold}
\usepackage{hyperref}

\numberwithin{equation}{section}
\begin{document}

\noindent {Asymptotics of multifractal products of spherical random fields\footnote[2]{Dedicated to the memory of Professor Vitalii Pavlovych Motornyi, corresponding member of National Academy of Sciences of Ukraine
 (1940-2025)}}
\vskip 3mm

\vskip 5mm
\noindent Illia Donhauzer\footnote[1]{ Institute of Mathematics for Industry, Kyushu University, Japan. La Trobe University, Melbourne, Australia.  emails: \href{donhauzer.illia.501@m.kyushu-u.ac.jp}{donhauzer.illia.501@m.kyushu-u.ac.jp}, \href{I.Donhauzer@latrobe.edu.au}{I.Donhauzer@latrobe.edu.au}}.

\vskip 3mm

\noindent ABSTRACT

The paper studies multifractal random measures on spheres $\mathbb{S}^d$ constructed using multifractal
products of random fields. The paper presents new limit theorems for multifractal
products of spherical fields and conditions for the non-degeneracy of the limiting measure.
We also focus on studying multifractal properties of the limiting measure and calculate its
R\'enyi function under conditions that are expressed in terms of second-order moments in the
geometric Gaussian scenario. Compared to earlier results, the obtained limit theorems hold under general mixing conditions that allow for
considering multifractal products of fields from a wide class and deriving random measures
with flexible multifractal properties.
\vskip 4mm

\renewcommand{\qedsymbol}{$\blacksquare$}

\theoremstyle{definition}
\newtheorem{definition}{Definition}

\theoremstyle{definition}
\newtheorem{assumption}{Assumption}

\theoremstyle{definition}
\newtheorem{lemma}{Lemma}

\theoremstyle{definition}
\newtheorem{theorem}{Theorem}

\theoremstyle{definition}
\newtheorem{remark}{Remark}

\theoremstyle{definition}
\newtheorem*{example}{Example}

\theoremstyle{definition}%
\newtheorem{corollary}{Corollary}%

\section{Introduction}
The scope of this article is the study of multifractal random measures on spherical domains. Multifractal random measures defined on spheres have found applications in the study of various phenomena. Among others, cosmic microwave background (CMB) radiation \cite{broadbridge2024stochastic, broadbridge2022multifractionality,  leonenko2021analysis},  distributions of galaxies in the universe \cite{martinez1990clustering}, and epicenters of earthquakes \cite{geilikman1990multifractal}. 

The concept of multifractality initially emerged in the context of physics. In contrast to monofractal systems, multifractals exhibit more complex local behaviour, which cannot be quantified by a single parameter. Across the literature, the definition of multifractality is framed through distinct mathematical formalisms: some authors define it via the path properties \cite{jaffard1999multifractal} (in the context of L\'evy processes), while others characterise it through the scaling behavior of moments or the properties of the R\'enyi function \cite{kahane1987, molchan1996}. For a discussion on the definitions of multifractality, see \cite{grahovac2018bounds}. In this paper, we characterise multifractal properties of random measures using the R\'enyi function.

The multifractal spectrum $D(h)$ provides the Hausdorff dimensions of the sets on which a signal exhibits a specific local Hurst exponent $h$, thereby quantifying its local singular behavior. The R\'enyi function is related to the multifractal spectrum via a Legendre transform \cite{jaerisch2020multifractal, seuret2017multifractal}. Because the R\'enyi function exhibits greater regularity with respect to the data and is often more tractable both analytically and numerically, it is a preferred tool for multifractal analysis \cite{riedi1995improved}.

Various approaches have been proposed for constructing multifractal random measures, including binomial cascading models \cite{kahane1987}, and Mandelbrot-type cascades \cite{molchan1996}, where, under simple conditions, the R\'enyi function and the multifractal spectrum were~derived. 

Increments of classical cascading measures form a non-stationary process. To overcome this issue, Mannersalo et al. \cite{Man} introduced a construction of random multifractal measures based on iterative multiplication (multifractal products) of stationary stochastic processes. Necessary and sufficient conditions for convergence in the space \(L_2\), and  the R\'enyi function of the limiting measure on the interval \([1,2]\) were obtained. Following the later approach, \cite{Anh27042009} studied multifractal products of stationary diffusion processes. Denisov and Leonenko \cite{denisov2016limit} extended these results by obtaining conditions for convergence in \(L_p,\  p \geq2,\) spaces, and evaluating the Rényi function on the interval \([0,p]\). Donhauzer and Olenko \cite{donhauzer2024limit} considered multifractal products of random fields and constructed multifractal random measures on hypercubes $[0,1]^d,\ d\geq1,$ deriving conditions for convergences in \(L_p,\  p \geq2,\) spaces, establishing rates of convergence, and calculating the Rényi function on the interval \([0,p]\). 

Motivated by the Planck mission’s CMB \cite{adam2016planck, ade2016planck} data, \cite{leonenko2021analysis} studied multifractal random measures on the sphere $\mathbb{S}^2$ and investigated important cases in which the R\'enyi function can be explicitly calculated. Their results demonstrated the applicability of the resulting multifractal spherical measures to cosmology and revealed intricate local scaling properties of CMB temperature fluctuations.

Our focus in this paper is to construct multifractal random measures on spheres $\mathbb{S}^{d}$. We derive multifractal random measures as limits of multifractal products of spherical random fields. Compared to earlier results, the key novelties of the paper are:
\begin{itemize}
\item deriving limit theorems for convergence of multifractal products of spherical fields in $L_p, \ p\geq2,$ spaces, which generalises existing results \cite{leonenko2013renyi} for $L_2$ convergence;
\item the obtained limit theorems hold under general mixing conditions extending the results in the literature where the exponential decay of covariances is commonly assumed \cite{leonenko2021analysis, leonenko2013renyi, Man};
\item calculating the R\'enyi function of the limiting measure;
\item demonstrating the obtained results in the geometric Gaussian scenario where the R\'enyi
function is calculated explicitly under simple conditions on second-order moments.
\end{itemize}
The paper is structured as follows: section \ref{sec1_mult} provides main definitions and results from the theory of spherical random fields, and introduces the model under study; section \ref{sec2_mult} derives limit theorems for multifractal products of spherical fields and conditions for the non-degeneracy of the limiting measure; section \ref{sec:mult} calculates the R\'enyi function of the obtained in section \ref{sec2_mult}  limiting measure.

For the convenience of the reader, we summarize the basic notation that we use
throughout the paper. Bold letters are used to denote vectors (e.g. $\bm{x}\in\mathbb{R}^{d+1}$), while numbers and scalar variables are denoted in a regular font. $\mathbb{N}_0$ denotes the set of non-negative integers $\{0,1,...\},$ $\mathbb{S}^d$, $d\geq 1$, denotes the centered unit hypersphere in $\mathbb{R}^{d+1}.$
The notation $(\eta, \boldsymbol{\theta}) = (\eta,\theta_1, \theta_2,..., \theta_{d}),\ \eta\geq0, \ \theta_i\in[0,\pi], \ i=\overline{1,d-1}, \ \theta_d\in[0,2\pi),$ is used to define points in the $d+1$-dimensional spherical system of coordinates, and the simplified notation  $\boldsymbol{\theta} = (\theta_1, \theta_2,..., \theta_{d}),$ is used for points on $\mathbb{S}^d,$ i.e. when $\eta=1.$ The north pole of $\mathbb{S}^d$ is denoted by $\mathbb{0}_d$ and has spherical coordinates $\mathbb{0}_d = (0,0,...,0).$ $||\cdot||$ and $d_{{\rm Euc}}(\cdot, \cdot)$ are the Euclidean norm and the Euclidean distance in $\mathbb{R}^{d+1}$ respectively, and $d_{\mathbb{S}^d}(\cdot, \cdot)$ denote the angular (great circle) distance on $\mathbb{S}^d.$ By \cite[Page~130]{MR0697386}, it holds
\begin{equation}\label{dist}d_{{\rm Euc}}(\bm{\theta}_1, \bm{\theta}_2) = 2\sin\left(\frac{d_{\mathbb{S}_d}(\bm{\theta}_1, \bm{\theta}_2)}{2}\right),\ \bm{\theta}_1, \bm{\theta}_2 \in \mathbb{S}^d.\end{equation} A spherical cap $B_u(\bm{\theta})$ on $\mathbb{S}^d$ with center $\bm{\theta} \in \mathbb{S}^d$ and angular radius $u \in [0, \pi]$ is defined as $B_u(\bm{\theta}) := \{ \bm{\theta}_1 \in \mathbb{S}^d : d_{\mathbb{S}^d}(\bm{\theta}_1, \bm{\theta}) \leq u \}.$ The Borel $\sigma$-algebra on $\mathbb{S}^d$ is denoted by $\mathcal{B}(S^d).$ The Lebesgue measure  on $\mathbb{S}^d$ is denoted by $\sigma(\cdot).$ $C$ with subscripts represents finite positive constants, that are not necessarily the same in each~appearance.

\section{Preliminaries}
\label{sec1_mult} In this section we describe the model under study and state auxiliary results. Let us first state technical Lemmas.

\begin{lemma}\label{lemma_as} Let $B_u(\bm{\theta}), \ \bm{\theta}\in\mathbb{S}^d, $ be a spherical cap on $\mathbb{S}^d.$ Then, it holds
\[ \sigma(B_u(\bm{\theta})) \sim C u^d, \ {\rm when} \ u\to 0, \] for some constant $C.$
\end{lemma}

The proof of the above Lemma is trivial.

\begin{lemma}\label{lemma_asymp}
Let $\bm{\theta}_1 =(\theta_1^{(1)},\theta_2^{(1)},...,\theta_d^{(1)}),$ and $ \bm{\theta}_2 = (\theta_1^{(2)}, \theta_2^{(2)},...,\theta_d^{(2)})$ be points on $\mathbb{S}^d$ such that $0\leq\theta_1^{(k)}\leq\frac{\pi}{2},\ k=1,2.$ Then, it holds
\[\frac{4 u}{\pi^2} d_{\mathbb{S}_d}(\bm{\theta}_1, \bm{\theta}_2) \leq d_{\mathbb{S}_d}((u\theta_1^{(1)},\theta_2^{(1)},...,\theta_d^{(1)}), (u\theta_1^{(2)},\theta_2^{(2)},...,\theta_d^{(2)}))\leq \frac{\pi^2 u}{4} d_{\mathbb{S}_d}(\bm{\theta}_1, \bm{\theta}_2)\] for all $u\in[0,1].$
\end{lemma} 

See the Appendix for the proof of the above Lemma.

Throughout the paper, all random variables and random fields are given on the fixed probability space $\big\{\Omega, \mathcal{F}, P\big\}.$  
\begin{assumption}
\label{assump_main_mult}
Let $ \Lambda(\bm{\theta}),\ \bm{\theta} \in \mathbb{S}^d,$ be a $\mathcal{B}(\mathbb{S}^d)\times\mathcal{F}$-measurable random field  such that $P(\Lambda(\bm{\theta})\geq0,\ \forall \bm{\theta}\in\mathbb{S}^d)=1,$ and $E\Lambda(\bm{\theta}) = 1,$ $E\Lambda^2(\bm{\theta})<+\infty, \ \bm{\theta}\in\mathbb{S}^d.$  
\end{assumption}

Let $\Lambda^{(i)}(\cdot), \ i\in \mathbb{N}_0,$ be a sequence of independent random fields satisfying Assumption \ref{assump_main_mult}. Consider the products 

$$\Lambda_n(\bm{\theta}) := \displaystyle\prod_{i=0}^{n} {\Lambda}^{(i)}(\bm{\theta}), $$ 
and the cumulative random measures 
\begin{equation}\label{meas}
\mu_n(B) := \int\limits_{B}\Lambda_n(\bm{\theta}) \sigma(d\bm{\theta}), 
\end{equation} where $B\in\mathcal{B}(\mathbb{S}^d)$ and $n\in \mathbb{N}_0.$

Let us show that for a fixed $B\in\mathcal{B}(\mathbb{S}^d)$  the sequence of random variables $\{\mu_n(B), \  n \in \mathbb{N}_0\}$ and $\sigma$-algebras $\mathcal{F}_n =\sigma\{\Lambda^{(0)}(\bm{\theta}),\Lambda^{(1)}(\bm{\theta}),...,\Lambda^{(n)}( \bm{\theta}),\ \bm{\theta}\in B\}, \ n\in \mathbb{N}_0,$ form a martingale. 	Indeed, by applying Tonelli's theorem and as $\Lambda^{(i)}(\cdot), \ i\in \mathbb{N}_0,$ are independent random fields, one gets for~$n\geq j$
\[E\left(\mu_n(B) |  \mathcal{F}_j \right) =  E\left(\int\limits_{B}\Lambda_n(\bm{\theta}) \sigma(d\bm{\theta}) \ \bigg| \  \mathcal{F}_j \right) = \int\limits_{B} E\left(\displaystyle\prod_{i=0}^{n} \Lambda^{(i)}(\bm{\theta}) \ \bigg| \ \mathcal{F}_j \right) \sigma(d\bm{\theta})\]
\[ = \int\limits_{B} \displaystyle\prod_{i=0}^{j} \Lambda^{(i)}(\bm{\theta}) \sigma(d\bm{\theta})  = \mu_j(B).\]  

Let $\mu(\cdot)$ be a random measure on $\mathbb{S}^d.$ The R\'enyi function \cite{leonenko2013renyi, Man} of $\mu(\cdot)$ is defined~as 

\[ \tau_{\mu}(q):=\liminf\limits_{l\to\infty} \frac{\log\left(\sum\limits_{\Delta_l^{(m)}\in \Delta_l} E \mu^q(\Delta_l^{(m)})\right)}{ \log |\Delta_l^{(0)}|},\ \ \ q\geq0, \]  where $\Delta_l$ denotes the mesh formed by $l$-level dyadic decomposition of the sphere $\mathbb{S}^d.$ Consult \cite{hytonen2012non, kaenmaki2012existence} for dyadic decompositions of metric spaces.

In the next section, we state limit theorems for cumulative random measures $\mu_n(\cdot)$ and derive conditions for the non-degeneracy of the limiting measure.

\section{Limit theorems for cumulated measures}
\label{sec2_mult}

Our first result provides conditions for the convergence of the cumulative random measures. Consider the mixing condition.

\begin{assumption}
\label{assump2_mult}
Let for some $k\geq2$ and a vector $\textit{\textbf{p}} = (p_1,p_2,...,p_k),\ p_j\geq0, \ j=\overline{1,k},$ there exist a non-increasing function $\rho(\cdot,\textit{\textbf{p}})$  and a constant $b>1$ such that
\begin{equation}\label{assump2_mult_ineq}\rho(b^i\min\limits_{m\neq h}d_{\mathbb{S}^d}(\bm{\theta}_m,\bm{\theta}_h),\textit{\textbf{p}}) \geq E\left(\prod_{j=1}^k (\Lambda^{(i)}(\bm{\theta}_j))^{p_j}\right),\end{equation} for all $i\in \mathbb{N}_0$  and $\bm{\theta}_j \in \mathbb{S}^d,\ j = \overline{1,k}.$
\end{assumption} Let us use the notation $\rho(\cdot, p),$ if Assumption \ref{assump2_mult} is satisfied for the vector $\textit{\textbf{p}} = (1,1,...,1)$ consisting of $p\geq2$ elements.

\begin{theorem}
\label{th2_mult}
Let random fields $\Lambda^{(i)}(\cdot),\ i\in \mathbb{N}_0,$ satisfy Assumption {\rm \ref{assump_main_mult}} and Assumption {\rm \ref{assump2_mult}} for $\ k\geq 2$ and the vector $\textit{\textbf{p}} = (p_1,p_2,...,p_k), \ p_j\geq1, \ j=\overline{1,k},$ such that $\sum_{j=1}^kp_j = p \geq2,$ 

\begin{equation}
\label{cond2.1_mult}
b>\rho^{\frac{1}{d}}(0, \textit{\textbf{p}} )
\end{equation} and

\begin{equation}
\label{cond2.2_mult}
\sum_{i=0}^\infty\ln\left(\rho(b^i, \textit{\textbf{p}})\right) < \infty.\end{equation}  If $\mathfrak{B} = \{B: \ B\in\mathcal{B}(\mathbb{S}^d) \}$ is a fixed finite or countable system of Borel sets, then, there exists a limiting measure $\mu(\cdot)$ such that  for all $\ B\in\mathfrak{B}$ it holds $\mu_n(B)\to \mu(B)$, when $\ n\to\infty,$ with probability $1$ and in the space $L_p.$
\end{theorem}

\begin{proof} As for a fixed $B\in\mathfrak{B},$  the sequence $\{\mu_n(B), n \geq 1\},$ is a martingale, it converges with probability 1 and in the space $L_p$ if $\sup_{n}E\mu_{n}^p(B)<\infty.$

Let us find the uniform in $n$ upper bound for the following moment
\[E\mu^p_n(B) =  E \prod_{j=1}^k\left(\int\limits_{B}\Lambda_n(\bm{\theta}_j)\sigma(d\bm{\theta}_j)\right)^{p_j}.\] For $p_i\geq1, \ i=\overline{1,k},$ by applying H\"{o}lder's inequality, one gets \[E\mu^p_n(B) \leq C \cdot E \left(\prod_{j=1}^k\int\limits_{B}\big(\Lambda_n(\bm{\theta}_j)\big)^{p_j} \sigma(d\bm{\theta}_j)\right) = C \cdot E\left( \int\limits_{B^k}  \prod_{j=1}^k \big(\Lambda_n(\bm{\theta}_j)\big)^{p_j} \prod_{j=1}^k\sigma(d\bm{\theta}_j) \right)\]

\begin{equation}\label{rho_mult} 
= C\cdot E \left( \int\limits_{B^k}\displaystyle\prod_{i=0}^{n}\prod_{j=1}^k \bigg( \Lambda^{(i)}(\bm{\theta}_j) \bigg)^{p_j} \prod_{j=1}^k \sigma(d\bm{\theta}_j)\right) \leq C \int\limits_{B^k} \prod_{i= 0}^{n} \rho(b^i\min_{m \neq h}d_{\mathbb{S}^d}(\bm{\theta}_m,\bm{\theta}_h), \textit{\textbf{p}} ) \prod_{j=1}^k \sigma(d\bm{\theta}_j),
\end{equation}  where the last estimate was obtained by applying Tonelli's theorem and Assumption \ref{assump2_mult}. 

By Assumption \ref{assump_main_mult} the random fields $\Lambda^{(i)}(\cdot), \ i\in\mathbb{N}_0,$ are non-negative with probability 1. Thus, by \eqref{assump2_mult_ineq} the function $\rho(\cdot,\bm{p})$ is non-negative, and it holds

\begin{equation}\label{eq:rho_est} \rho(b^i\min_{m \neq h}d_{\mathbb{S}^d}(\bm{\theta}_m,\bm{\theta}_h), \textit{\textbf{p}} )  \leq \sum\limits_{\substack{m,h: m\neq h}}  \rho(b^id_{\mathbb{S}^d}(\bm{\theta}_m,\bm{\theta}_h), \textit{\textbf{p}} ).\end{equation} By applying \eqref{eq:rho_est} one estimates \eqref{rho_mult} as
\[ C \sum\limits_{\substack{m,h: m\neq h}} \int\limits_{B}\int\limits_{B} \prod_{i= 0}^{n} \rho(b^id_{\mathbb{S}^d}(\bm{\theta}_m,\bm{\theta}_h), \textit{\textbf{p}} ) \sigma(d\bm{\theta}_m) \sigma(d\bm{\theta}_h).\] The integrals in the above sum are identical. Thus, it equals
\begin{equation}\label{th1:eq_double_int} C \int\limits_{B} \int\limits_{B} \prod_{i= 0}^{n} \rho(b^id_{\mathbb{S}^d}(\bm{\theta}_1,\bm{\theta}_2), \textit{\textbf{p}} )\sigma(d\bm{\theta}_1) \sigma(d\bm{\theta}_2) \leq C \int\limits_{\mathbb{S^d}} \int\limits_{\mathbb{S^d}} \prod_{i= 0}^{n} \rho(b^id_{\mathbb{S}^d}(\bm{\theta}_1,\bm{\theta}_2), \textit{\textbf{p}} ) \sigma(d\bm{\theta}_1) \sigma(d\bm{\theta}_2).\end{equation} As the integrand on the right-hand side of \eqref{th1:eq_double_int} depends only on the distance $d_{\mathbb{S}^d}(\bm{\theta}_1,\bm{\theta}_2)$ and the inner integration is over $\mathbb{S}^d,$ one can see that the inner integral is constant for all $\bm{\theta}_2.$ Thus, \eqref{th1:eq_double_int} equals
\[ C \int\limits_{\mathbb{S^d}} \prod_{i= 0}^{n} \rho(b^id_{\mathbb{S}^d}(\bm{\theta}_1,\mathbb{0}_d), \textit{\textbf{p}}) \sigma(d\bm{\theta}_1).\]
As for $\bm{\theta}_1 = (\theta_1,\theta_2,...,\theta_d)$ it holds $d_{\mathbb{S}^d}(\bm{\theta}_1,\mathbb{0}_d) = \theta_1,$ the above equals to
\[ C \int\limits_{[0,\pi]^{d-1}\times[0,2\pi)} \prod_{i= 0}^{n} \rho\left(b^i\theta_1, \textit{\textbf{p}} \right) \sin^{d-1}\theta_1\sin^{d-2}\theta_2\cdot...\cdot\sin\theta_{d-1} \prod\limits_{j=1}^d d\theta_j \] \begin{equation}\label{int_prod} \leq C \int\limits_{0}^\pi \prod_{i= 0}^{n} \rho\left(b^i\theta_1, \textit{\textbf{p}} \right) \sin^{d-1}\theta_1 d\theta_1. \end{equation} Since $\rho(\cdot, \bm{p})$ is non-increasing and $b>1,$ the following product can be estimated as
\[ \prod_{i= 1}^{n} \rho\left(b^i\theta_1, \textit{\textbf{p}} \right) = e^{\sum\limits_{i=1}^{n} \ln\rho\left(b^i\theta_1, \textit{\textbf{p}} \right)}\leq e^{\int\limits_0^{n}\ln\rho\left(b^x\theta_1,\textit{\textbf{p}}  \right)dx}.\] Thus, \eqref{int_prod} is bounded by
\begin{equation}\label{int_int} C \int\limits_{0}^\pi  \exp\left({\int\limits_0^{n}\ln\rho\left(b^x\theta_1,\textit{\textbf{p}}  \right)dx}\right) \sin^{d-1}\theta_1 d\theta_1 = C \int\limits_{0}^\pi \exp\left({\int\limits_{\log_b \theta_1}^{n + \log_b \theta_1}\ln\rho\left( b^u,\textit{\textbf{p}}  \right)du} \right) \sin^{d-1}\theta_1  d\theta_1,\end{equation} where the last is obtained by the change of variables $x = u-\log_b \theta_1.$ Let us investigate separately the following two cases.

If there exists $u_0$ such that $\rho(b^{u_0},\textit{\textbf{p}}) = 1,$ then, $\rho(b^u,\textit{\textbf{p}}) \leq 1,\ u\geq u_0,$ as $\rho(\cdot,\textit{\textbf{p}})$ is non-increasing, and \eqref{int_int} is bounded by
\begin{equation}\label{int_est} C \int\limits_{0}^\pi  \exp\left({\int\limits_{\log_b \theta_1}^{u_0}\ln\rho\left( b^u,\textit{\textbf{p}}  \right)du}\right) \sin^{d-1}\theta_1 d\theta_1.\end{equation} Using the identity $a^{\log_bc}=c^{\log_ba}$ and as $\sin \theta_1 \leq\theta_1,$ one obtains that the above integrand is bounded by
\[ (\rho(0, \textit{\textbf{p}}))^{u_0-\log_b \theta_1 } \sin^{d-1}\theta_1  = C \theta_1^{-\log_b\rho(0,\textit{\textbf{p}})} \sin^{d-1}\theta_1 \leq C \theta_1^{d-1-\log_b\rho(0,\textit{\textbf{p}})}.\] The above function is integrable on $[0,\pi]$ if $d-1-\log_b\rho(0, \textit{\textbf{p}})>-1,$ which is equivalent to $b>\rho^{\frac{1}{d}}(0, \textit{\textbf{p}} ).$ 

Now let us study the case when $\rho(b^u,\textit{\textbf{p}}) > 1$ for all $u,$ then \eqref{int_int} is bounded by
\[ C \int\limits_{0}^\pi \exp\left({\int\limits_{\log_b \theta_1 }^{\infty}\ln\rho\left( b^u,\textit{\textbf{p}}  \right)du}\right)\sin^{d-1}\theta_1 d\theta_1 \]
\begin{equation}\label{int_est_2} = C \int\limits_{0}^\pi  \exp\left({\int\limits_{\log_b \theta_1 }^{0}\ln\rho\left( b^u,\textit{\textbf{p}}  \right)du} + {\int\limits_{0}^{\infty}\ln\rho\left( b^u,\textit{\textbf{p}}  \right)du}\right) \sin^{d-1}\theta_1 d\theta_1.\end{equation} The second integral in the exponent is bounded if and only if 
$\sum\limits_{i = 0}^\infty\ln\left(\rho(b^i,\textbf{\textit{p}})\right)<\infty,$ which is the condition \eqref{cond2.2_mult}. Thus, the integrand in \eqref{int_est_2} is bounded by
\[ C(\rho(0, \textit{\textbf{p}}))^{-\log_b \theta_1}  \sin^{d-1}\theta_1= C  \theta_1^{d-1-\log_b\rho(0,\textit{\textbf{p}})},\] which is integrable under the same condition as in the first case. 

Thus, there exists a constant majorizing $E\mu_{n}^p(B)$ for all $n,$ and $\mu_n(B)\to\mu(B), \ n\to\infty,$ with probability $1$ and in the space $L_p.$ Finally, as the system of sets $\mathfrak{B}$ is at most countable, the convergence with probability~$1$ holds for the whole system $\mathfrak{B}.$ \end{proof}

A partial case is the result about $L_p, \ p\geq2, \ p\in\mathbb{N},$ convergence.

\begin{corollary}
Let random fields $\Lambda^{(i)}(\cdot),\ i\in \mathbb{N}_0,$ satisfy Assumption {\rm \ref{assump_main_mult}} and Assumption {\rm \ref{assump2_mult}} for the vector $\textit{\textbf{p}} = (1,1,..,1)$ consisting of $p$ elements, such that 

\begin{equation*}
b>\max\limits_{i\in\mathbb{N}_0,\ \bm{\theta}\in\mathbb{S}^d}\left(E(\Lambda^{(i)}(\bm{\theta}))^p\right)^\frac{1}{d}
\end{equation*} and

\begin{equation*}
\sum_{i=0}^\infty\ln\left(\rho(b^i, p)\right) < \infty.\end{equation*}  If $\mathfrak{B} = \{B: \ B\in\mathcal{B}(\mathbb{S}^d) \}$ is a fixed finite or countable system of Borel sets, then, there exists a limiting measure $\mu(\cdot)$ such that  for all $\ B\in\mathfrak{B}$ it holds $\mu_n(B)\to \mu(B)$, when $\ n\to\infty,$ with probability $1$ and in the space $L_p.$

\end{corollary}

Next, let us establish conditions for the non-degeneracy of the obtained in Theorem~\ref{th2_mult} limiting measure $\mu(\cdot).$ 

\begin{assumption}
\label{assump3_mult} Let there exist a point $\widetilde{\bm{\theta}}\in\mathbb{S}^d,$ a vector $\textit{\textbf{p}} = (p_1,p_2,...,p_k),\ p_j\geq0, \ j=\overline{1,k}, \ k\geq2,$ and a positive non-increasing function $\upupsilon(\cdot,\textit{\textbf{p}})$  such that
\[\upupsilon(b^ix, \textbf{\textit{p}}) \leq \min\limits_{\substack{\bm{\theta}_j\in B_x(\widetilde{\bm{\theta}}) \\ j=\overline{1,k}}}E\bigg(\prod_{j=1}^k (\Lambda^{(i)}(\bm{\theta}_j))^{p_j}\bigg),\] for all $i\in\mathbb{N}_0,$ and $x\in[0,\pi],$ where $b>1$ is the same as in Assumption {\rm \ref{assump2_mult}}.
\end{assumption} 
Let us use the notation $\upupsilon(\cdot, p)$  if $\textbf{\textit{p}}$ is a unit vector consisting of $p$ elements.

\begin{remark} \label{remark1}
If the random fields $\Lambda^{(i)}(\cdot), \ i\in\mathbb{N}_0,$ are strongly isotropic and satisfy Assumption~{\rm \ref{assump3_mult}} for some point $\widetilde{\bm{\theta}}\in\mathbb{S}^d$, a vector $\bm{p}$ and a function $\upupsilon(\cdot,\textbf{\textit{p}}),$ then Assumption {\rm \ref{assump3_mult}} is satisfied for all $\bm{\theta}\in\mathbb{S}^d$ with the vector $\bm{p}$ and the function $\upupsilon(\cdot,  \textbf{\textit{p}}).$ 
\end{remark}

The following Theorem provides conditions for the non-degeneracy of the limiting measure~$\mu(\cdot)$.

\begin{theorem}\label{th3_mult} Let the conditions of Theorem {\rm \ref{th2_mult}} be satisfied, i.e. $\mu(B)\in L_p, \ p\geq2,$ for $B\in \mathfrak{B},$ and Assumption {\rm \ref{assump3_mult}} holds true for some $\widetilde{\bm{\theta}}\in\mathbb{S}^d,$ $k$-dimensional vector $\widetilde{\textbf{\textit{p}}}=(q/k,...,q/k), \ q\in[1,p],$ and function $\upupsilon(\cdot, \widetilde{\textbf{\textit{p}}}).$ If
\begin{equation}\label{cond:nondeg} \sum\limits_{i=0}^\infty \ln\left( \frac{ E\left(\Lambda^{({i})}(\widetilde{\bm{\theta}})\right)^q}{\upupsilon(b^{-i}, \widetilde{\textbf{\textit{p}}})} \right)<\infty,\end{equation} then, the limiting measure $\mu(\cdot)$ is non-degenerate, and it holds
\begin{equation}\label{eq:low_bound} E\mu^q(B_u(\widetilde{\bm{\theta}}))\geq C u^{dq-\log_b\sigma_q},\ {\rm when} \ u\to0,
\end{equation} where $\sigma_q=\min\limits_{i\in\mathbb{N}_0} E\left(\Lambda^{(i)}(\widetilde{\bm{\theta}})\right)^q.$
\end{theorem}

\begin{remark}\label{remark_eventual} The inequality \eqref{eq:low_bound} is understood in the following sense:  there exist constants $C>0$ and $u_0>0$ such that $ E\mu^q(B_u(\widetilde{\bm{\theta}}))\geq C u^{dq-\log_b\sigma_q}$ for all $u\in(0,u_0).$
\end{remark}

\begin{proof} Without loss of generality, assume that $\widetilde{\bm{\theta}} = \mathbb{0}_d.$ Let us derive the lower bound for $E\mu^q(B_u(\mathbb{0}_d))$ when $u\to0.$

By using the reverse H\"older's inequality \cite[Page 140]{hardy1934inequalities} $||f_1f_2||_1 \geq ||f_1||_{\widetilde{p}}||f_2||_{\widetilde{q}},\ \frac{1}{\widetilde{p}}+\frac{1}{\widetilde{q}}=1, \ \widetilde{p}\in(0,1),\ \widetilde{q}<0,$ one gets
\begin{equation}\label{holder_upper} ||f_1f_2||_1^q = ||f_1|f_2|^{1/\widetilde{p}}|f_2|^{1/\widetilde{q}}||_1^q \geq ||f_1|f_2|^{1/\widetilde{p}}||_{\widetilde{p}}^q|||f_2|^{1/\widetilde{q}}||_{\widetilde{q}}^q = |||f_1|^{\widetilde{p}}f_2||_1^{q/\widetilde{p}}||f_2||_1^{q/\widetilde{q}}.\end{equation} Application of the above inequality for $\widetilde{p}$ such that $q/\widetilde{p}=k\in\mathbb{N}$ and $f_1(\bm{\theta}) = \Lambda^{(n)}(\bm{\theta}),$ $f_2(\bm{\theta})=\Lambda_{n-1}(\bm{\theta})$ results in
\[ E\mu_{n}^q(B_u(\mathbb{0}_d)) = E\left(\int\limits_{B_u(\mathbb{0}_d)}\Lambda^{(n)}(\bm{\theta}) \Lambda_{n-1}(\bm{\theta}) \sigma(d\bm{\theta})\right)^q\]
\[ \geq  E\left( \int\limits_{B_u(\mathbb{0}_d)} (\Lambda^{(n)}(\bm{\theta}))^{\widetilde{p}}\Lambda_{n-1}(\bm{\theta}) \sigma(d\bm{\theta})\right)^{q/\widetilde{p}}\left(\int\limits_{B_u(\mathbb{0}_d)} \Lambda_{n-1}(\bm{\theta})\sigma(d\bm{\theta}) \right)^{q/\widetilde{q}}\]
\[ = \int\limits_{\left( B_u(\mathbb{0}_d)\right)^k} E\left(\prod\limits_{j=1}^k(\Lambda^{(n)}(\bm{\theta}_j))^{\widetilde{p}}\right) E\left( \prod_{j=1}^k\Lambda_{n-1}(\bm{\theta}_j) \left(\int\limits_{B_u(\mathbb{0}_d)} \Lambda_{n-1}(\bm{\theta}')\sigma(d\bm{\theta}') \right)^{q/\widetilde{q}}\right) \prod_{j=1}^k\sigma(d\bm{\theta}_j)\]
\[ \geq \min\limits_{\substack{\bm{\theta}_j\in B_u(\mathbb{0}_d) \\ j=\overline{1,k}}}E\bigg(\prod_{j=1}^k (\Lambda^{(n)}(\bm{\theta}_j))^{\widetilde{p}}\bigg) \] \[\times E\left( \int\limits_{\left( B_u(\mathbb{0}_d)\right)^k}\prod_{j=1}^k\Lambda_{n-1}(\bm{\theta}_j) \left(  \int\limits_{B_u(\mathbb{0}_d)} \Lambda_{n-1}(\bm{\theta}')\sigma(d\bm{\theta}')\right)^{q/\widetilde{q}} \prod_{j=1}^k\sigma(d\bm{\theta}_j)\right) \]
\begin{equation}\label{eq:non_deg_est} =  \min\limits_{\substack{\bm{\theta}_j\in B_u(\mathbb{0}_d) \\ j=\overline{1,k}}}E\left(\prod_{j=1}^k (\Lambda^{(n)}(\bm{\theta}_j))^{\widetilde{p}}\right)  E\left(\int\limits_{B_u(\mathbb{0}_d)} \Lambda_{n-1}(\bm{\theta}) \sigma(d\bm{\theta})\right)^q.\end{equation}
Due to Assumption \ref{assump3_mult}, it holds
\[ \upupsilon(b^{n}u, \widetilde{\textbf{\textit{p}}}) \leq \min\limits_{\substack{\bm{\theta}_j\in B_u(\mathbb{0}_d) \\ j=\overline{1,k}}}E\left(\prod_{j=1}^k (\Lambda^{(n)}(\bm{\theta}_j))^{\widetilde{p}}\right).\] By applying the above in \eqref{eq:non_deg_est}, one gets
\begin{equation}\label{eq:non_deg_est2} E\mu_{n}^q(B_u(\mathbb{0}_d))\geq \upupsilon(b^{n}u, \widetilde{\textbf{\textit{p}}}) E\mu_{n-1}^q(B_u(\mathbb{0}_d))\geq \frac{\upupsilon(b^{n}u, \widetilde{\textbf{\textit{p}}}) \sigma_q}{E\left(\Lambda^{({n-1})}(\mathbb{0}_d)\right)^q} E\mu_{n-1}^q(B_u(\mathbb{0}_d)).\end{equation}

Let $n_u=[-\log_bu]$ be the largest integer such that $n_u\leq-\log_bu.$ As $\upupsilon(\cdot, \widetilde{\textbf{\textit{p}}})$ is non-increasing, from \eqref{eq:non_deg_est2} one obtains
\[E\mu_{n}^q(B_u(\mathbb{0}_d))\geq \frac{\upupsilon(b^{n-n_u}, \widetilde{\textbf{\textit{p}}})\sigma_q}{E\left(\Lambda^{({n-1})}(\mathbb{0}_d)\right)^q} E\mu_{n-1}^q(B_u(\mathbb{0}_d)).\]
Iteration of the above inequality results in
\[  E\mu_{n_u-1}^q(B_u(\mathbb{0}_d))\geq E\mu_1^q(B_u(\mathbb{0}_d)) \sigma_q^{n_u-2} \prod_{i=2}^{n_u-1} \frac{\upupsilon(b^{i-n_u}, \widetilde{\textbf{\textit{p}}})}{E\left(\Lambda^{({i-1})}(\mathbb{0}_d)\right)^q}\] \begin{equation}\label{eq:non_deg_est3}=E\mu_1^q(B_u(\mathbb{0}_d)) \sigma_q^{n_u-2} \prod_{i=1}^{n_u-2} \frac{\upupsilon(b^{-i}, \widetilde{\textbf{\textit{p}}})}{E\left(\Lambda^{({i})}(\mathbb{0}_d)\right)^q}.\end{equation} 

Since for $q\geq1$ it holds $E\mu_{1}^q(B_u(\mathbb{0}_d)) \geq (E\mu_{1}(B_u(\mathbb{0}_d)))^q$, and as by Lemma \ref{lemma_as} $\sigma(B_u(\mathbb{0}_d))\sim Cu^d,\ u\to0,$ one gets that $E\mu_{1}^q(B_u(\mathbb{0}_d))\geq C u^{dq}, \ u\to0.$  As by Assumption~\ref{assump3_mult} 
\[\upupsilon(b^{-i}, \widetilde{\textbf{\textit{p}}})\leq E\left(\Lambda^{({i})}(\mathbb{0}_d)\right)^q, \ i\in\mathbb{N}_0,\] and it holds
\[ \sigma_q^{n_u} = \sigma_q^{-\log_bu+\log_bu+[-\log_bu]}\geq C \sigma_q^{-\log_bu} = C u^{-\log_b\sigma_q},\] one obtains from \eqref{eq:non_deg_est3}
\[ E\mu_{n_u-1}^q(B_u(\mathbb{0}_d))\geq C u^{dq-\log_b\sigma_q} \prod_{i=1}^{\infty} \frac{\upupsilon(b^{-i}, \widetilde{\textbf{\textit{p}}})}{E\left(\Lambda^{({i})}(\mathbb{0}_d)\right)^q}\geq C u^{dq-\log_b\sigma_q},\ {\rm when} \ u\to0,\] where the last holds due to \eqref{cond:nondeg}. Since $\{\mu_n^q(B_u(\mathbb{0}_d)), \ n\geq 1\},\ q~\geq~1,$ is a submartingale, it holds $E\mu^q(B_u(\mathbb{0}_d))\geq E\mu_{n_u-1}^q(B_u(\mathbb{0}_d)),$ which finishes the proof of Theorem. \end{proof}

Let us consider the geometric Gaussian scenario when the convergence and non-degeneracy conditions of the limiting measure $\mu(\cdot)$ are expressed in terms of the second-order~moments.

\begin{assumption}\label{eq:log_Gauss}
Let $X^{(i)}(\bm{u}), \ \bm{u}\in\mathbb{R}^{d+1}, \ i\in\mathbb{N}_0,$ be a family of independent identically distributed homogeneous isotropic Gaussian random fields with zero means, variances $s^2>0,$ and covariances $r(||\bm{u}_1-\bm{u}_2||):=EX^{(i)}(\bm{u_1})X^{(i)}(\bm{u_2}), \ \bm{u_1},\bm{u_2}\in\mathbb{R}^{d+1}, \ i\in\mathbb{N}_0.$ Let
\begin{equation*}
\Lambda^{(i)}(\bm{\theta}) = \frac{e^{X^{(i)}(b^i, \bm{\theta})}}{Ee^{X^{(i)}(b^i, \bm{\theta})}},\ i\in\mathbb{N}_0, \ b>1, \ \bm{\theta}\in\mathbb{S}^d,
\end{equation*} where the notation $(b, \bm{\theta}), \ b\geq 0, \ \bm{\theta}\in\mathbb{S}^d,$ is used to denote points in the $d+1$-dimensional spherical system of coordinates.  
\end{assumption}

\begin{corollary}\label{cor_1}
Let $\Lambda^{(i)}(\cdot), \ i\in\mathbb{N}_0,$ satisfy Assumption {\rm \ref{eq:log_Gauss}} such that the covariance function $r(\cdot)$ is non-increasing, for an integer $p\geq2$ it holds

\begin{equation}\label{cond_gau_1} b>  \exp\left(\frac{p(p-1)s^2}{2d}\right),\end{equation} and
\begin{equation}\label{cond_gau_2} \sum\limits_{i=0}^\infty r( b^i) < \infty.\end{equation} If $\mathfrak{B} = \{B: \ B\in\mathcal{B}(\mathbb{S}^d)\}$ is a fixed finite or countable system of Borel sets, then, there exists a limiting measure $\mu(\cdot)$ such that for all $\ B\in\mathfrak{B}$ it holds $\mu_n(B)\to \mu(B)$, as $\ n\to\infty,$ with probability $1$ and in the space $L_p.$

If, in addition,

\begin{equation}\label{cond_gau_4} \sum\limits_{i=0}^\infty \left(s^2 - r\left(b^{-i} \right)\right)<\infty,\end{equation} then the limiting measure $\mu(\cdot)$ is non-degenerate.\end{corollary}
\begin{proof}
Let us show that Assumption {\rm\ref{assump2_mult}} is satisfied by estimating from above the expectation in~\eqref{assump2_mult_ineq}. Let $k\geq2$ and $\textit{\textbf{p}} = (p_1,p_2,...,p_k), \ p_i>0, \ i=\overline{1,k},$ such that $\sum_{j=1}^kp_j = p.$
As $X^{(i)}(\cdot), \ i\in\mathbb{N}_0,$ are homogeneous isotropic Gaussian fields on $\mathbb{R}^{d+1}$, the random variable $e^{\sum_{j=1}^k p_j X^{(i)}(b^i,\bm{\theta}_j)},$ $\bm{\theta}_j\in\mathbb{S}^d,\ j=\overline{1,k},$ has log-Gaussian distribution, and it~holds 
\[E\bigg(\prod_{j=1}^k (\Lambda^{(i)}(\bm{\theta}_j))^{p_j}\bigg) = \frac{1}{\prod\limits_{j=1}^k (Ee^{X^{(i)}(b^i,\bm{\theta}_j)})^{p_j}} Ee^{\sum\limits_{j=1}^k p_jX^{(i)}(b^i, \bm{\theta}_j)}=\frac{1}{e^{\frac{ps^2}{2}}} e^{{\frac{1}{2}E\left(\sum\limits_{j=1}^k p_j X^{(i)}(b^i,\bm{\theta}_j)\right)^2}}\] 
\[ =\frac{1}{e^{\frac{ps^2}{2}}} \exp\left(\frac{1}{2}\left( s^2 \sum_{j=1}^kp^2_j + \sum\limits_{l\neq v}p_lp_vr(d_{\rm Euc}((b^i,\bm{\theta}_l),(b^i,\bm{\theta}_v))\right) \right) \]
\begin{equation}\label{moment_mult} =\exp\left(\frac{1}{2}\left( s^2 \left(\sum_{j=1}^kp^2_j -p\right)  + \sum\limits_{l\neq v}p_lp_vr(b^id_{\rm Euc}(\bm{\theta}_l,\bm{\theta}_v))\right) \right),\end{equation} where the latter follows from $d_{\rm Euc}((\eta,\bm{\theta}_l),(\eta,\bm{\theta}_v)) = \eta\cdot d_{\rm Euc}(\bm{\theta}_l,\bm{\theta}_v),$ $\eta\geq0,\ \bm{\theta}_l,\bm{\theta}_v\in\mathbb{S}^d.$ As $r(\cdot)$ is non-increasing, the above is bounded by
\[
\exp\left(\frac{1}{2}\left( s^2 \left(\sum_{j=1}^kp^2_j -p\right) + \sum\limits_{l\neq v}p_lp_vr(b^i\min\limits_{m\neq h}d_{\rm Euc}(\bm{\theta}_m,\bm{\theta}_h))\right) \right).\]  As $\frac{2}{\pi}d_{\mathbb{S}^d}(\bm{\theta}_m,\bm{\theta}_h)\leq d_{\rm Euc}(\bm{\theta}_m,\bm{\theta}_h), \ \bm{\theta}_m,\bm{\theta}_h\in\mathbb{S}^d,$ from the latter follows that Assumption~{\rm \ref{assump2_mult}} is satisfied for the~function \begin{equation}\label{rho_gaus}\rho(\cdot,\textit{\textbf{p}} )=\exp\left(\frac{1}{2}\left( s^2 \left(\sum_{j=1}^kp^2_j -p\right) + \sum\limits_{l\neq v}p_lp_vr\left(\frac{2}{\pi}\ \cdot\right)\right) \right).\end{equation} By choosing the vector $\textbf{\textit{p}}$ to be a unit vector consisting of $p$ elements, one can see 
\[ \rho(0,p) =  \exp\left( \frac{p(p-1)s^2}{2} \right),\] and 
\[ \sum\limits_{i=0}^\infty\ln\left( \rho(b^i, p) \right) = \frac{p(p-1)}{2} \sum\limits_{i=0}^\infty r\left( \frac{2b^i}{\pi}\right).\] The above series converges, if the series in \eqref{cond_gau_2} converges. Thus, the assumptions of Theorem \ref{th2_mult} are satisfied if \eqref{cond_gau_1} and \eqref{cond_gau_2} hold true.

Now, let us show that the conditions of Theorem \ref{th3_mult} are satisfied such that the limiting measure $\mu(\cdot)$ is non-degenerate. Let $\bm{\theta}_j \in B_{x}(\mathbb{0}_d),\ x\in[0,\pi], \ j=\overline{1,k}, \ k\geq2.$ As $d_{\mathbb{S}^d}(\bm{\theta}_l,\bm{\theta}_v) \geq d_{\rm Euc}(\bm{\theta}_l,\bm{\theta}_v), \ l,v=\overline{1,k},$ one obtains $d_{\rm Euc}(\bm{\theta}_l,\bm{\theta}_v)\leq 2x, \ l,v=\overline{1,k},$ and as $r(\cdot)$ is non-increasing, from \eqref{moment_mult} it follows
\[ \min\limits_{\substack{\bm{\theta}_j\in B_x(\mathbb{0}_d) \\ j=\overline{1,k}}}E\bigg(\prod_{j=1}^k (\Lambda^{(i)}(\bm{\theta}_j))^{p_j}\bigg) \geq  \exp\left(\frac{1}{2}\left( s^2\left( \sum_{j=1}^kp^2_j -p\right) +  \sum\limits_{l\neq v}p_lp_vr(2b^i x)\right) \right),\] for all $i\in\mathbb{N}_0.$
Thus, Assumption {\rm \ref{assump3_mult}} is satisfied for the function 
\begin{equation}\label{upepsilon}\upupsilon(\cdot, \textbf{\textit{p}})= \exp\left(\frac{1}{2}\left( s^2\left(\sum_{j=1}^kp^2_j -p\right) + \sum\limits_{l\neq v}^kp_lp_vr\left(2\ \cdot\right)\right) \right).\end{equation} By choosing $\textbf{\textit{p}}$ to be the unit vector consisting of $p$ elements, one obtains
\[\upupsilon(b^{-i}, p)= \exp\left(\frac{p(p-1)r(2b^{-i})}{2} \right),\] and as \[E\left(\Lambda^{({i})}(\mathbb{0}_d)\right)^p = \exp\left( \frac{p(p-1)s^2}{2} \right),\ i\in\mathbb{N}_0,\] one can easily see that 

\[ \sum\limits_{i=0}^\infty\ln\left( \frac{E\left(\Lambda^{({i})}(\mathbb{0}_d)\right)^p}{\upupsilon(b^{-i}, p)} \right) = \frac{p(p-1)}{2}\sum\limits_{i=0}^\infty(s^2-r(2b^{-i})).\] The above series converges, if the series in \eqref{cond_gau_4} converges.  Thus, the assumptions of Theorem \ref{th3_mult} hold true if \eqref{cond_gau_4} is satisfied. \end{proof}

In the next section, we calculate the R\'enyi function of the limiting measure $\mu(\cdot).$

\section{R\'enyi function of the limiting measure}
\label{sec:mult}
In what follows, we assume that the following Assumption on moments is satisfied.

\begin{assumption}\label{assump_isot}
Let $\Lambda^{(i)}(\bm{\theta}),\ \bm{\theta}\in\mathbb{S}^d, \ i\in\mathbb{N}_0,$ be an infinite collection of independent strongly isotropic random fields such that $E(\Lambda^{(i)}(\mathbb{0}_d))^q = E(\Lambda^{(j)}(\bm{\theta}))^q,\ \bm{\theta}\in\mathbb{S}^d,\ q\in[0,p], $ $ p\geq 2, \ i,j\in~\mathbb{N}_0.$
\end{assumption}

To calculate the R\'enyi function of the limiting measure $\mu(\cdot),$ we need the following auxiliary result on its moments.

\begin{lemma}\label{lemma:mom_estim}
Let the conditions of Theorem {\rm \ref{th2_mult}} and Assumption {\rm \ref{assump_isot}}  be satisfied such that $\mu(B)\in L_p,\ B\in \mathfrak{B},$ for $p\geq2.$ If, for $\boldsymbol{q}=(q-1,1)$ with $q\in(0,1)$, the function $\rho(\cdot,\boldsymbol{q})$ is non-decreasing and satisfies
\begin{equation}\label{lemma:mom_estim_cond}
\sum_{i=0}^\infty \ln\left(\frac{\rho(b^{-i},\boldsymbol{q})}{E\bigl(\Lambda^{(i)}(\mathbb{0}_d)\bigr)^q}\right) < \infty,
\end{equation}
and if, for $q\in[1,p]$, the conditions of Theorem~{\rm\ref{th3_mult}} are satisfied, then there exist constants $C_1,C_2>0$ such that the following holds

\begin{equation}\label{ineq:mom} C_1 u^{dq - \log_bE(\Lambda^{(0)}(\mathbb{0}_d))^q} \leq E \mu^q(B_{u}(\bm{\theta})) \leq C_2 u^{dq - \log_bE(\Lambda^{(0)}(\mathbb{0}_d))^q},\ u\to0,\end{equation} for all $\ \bm{\theta}\in\mathbb{S}^d.$  

\begin{remark}
Analogously to \eqref{eq:low_bound}, \eqref{ineq:mom} is understood in the following sense: there exist constants $C_1, C_2 > 0$ and $u_0 > 0$ such that the upper and lower bounds hold for all $u \in (0,u_0)$.
\end{remark}

 \end{lemma}

\begin{proof} Without loss of generality, let us put $\bm{\theta} = \mathbb{0}_d.$ Let us find the upper bound for $E\mu_n^q(B_u(\mathbb{0}_d)),$ when $q\geq1.$ Let $n_u=[-\log_bu]$ be the largest integer such that $n_u\leq-\log_bu.$ Application of H\"older's inequality results in
\[\mu_n^q(B_u(\mathbb{0}_d))= \left( \int\limits_{B_u(\mathbb{0}_d)} \prod_{i=0}^{n} \Lambda^{(i)}(\bm{\theta})\sigma(d\bm{\theta}) \right)^q \]
\[\leq \left( \int\limits_{B_u(\mathbb{0}_d)} \left( \prod_{i=0}^{n_{\textit{\textbf{u}}}} \Lambda^{(i)}(\bm{\theta})\right)^q \prod_{i = n_{\textit{\textbf{u}}}+1}^{n}\Lambda^{(i)}(\bm{\theta})  \sigma(d\bm{\theta})\right) \left(\int\limits_{B_u(\mathbb{0}_d)} \prod_{i = n_{\mathfrak{u}}+1}^{n}\Lambda^{(i)}(\bm{\theta})  \sigma(d\bm{\theta})\right)^{q/p}.\]
By taking expectations from both sides, one obtains that $E\mu_n^q(B_u(\mathbb{0}_d))$ is bounded by \[ \int\limits_{B_u(\mathbb{0}_d)}\prod_{i=0}^{n_{u}}E (\Lambda^{(i)}(\bm{\theta}))^q  E\left( \left(\prod_{i = n_{u}+1}^{n} \Lambda^{(i)}(\bm{\theta} ) \right) \left( \int\limits_{B_u(\mathbb{0}_d)} \prod_{i = n_{u}+1}^{n} \Lambda^{(i)}(\bm{\theta}' ) \sigma(d\bm{\theta}')\right)^{q/p}  \right) \sigma(d\bm{\theta}).\] Therefore, by Assumption \ref{assump_isot} it holds
\[ E\mu_n^q(B_u(\mathbb{0}_d)) \leq (E(\Lambda^{(0)}(\mathbb{0}_d))^q)^{n_u+1}E\left( \int\limits_{B_u(\mathbb{0}_d)} \prod_{i = n_{u}+1}^{n} \Lambda^{(i)}(\bm{\theta}) \sigma(d\bm{\theta})\right)^{1+q/p} \]
\begin{equation}\label{lemma3_upper} =  (E(\Lambda^{(0)}(\mathbb{0}_d))^q)^{n_u+1} u^{dq} E\left(u^{-d} \int\limits_{B_u(\mathbb{0}_d)} \prod_{i = n_{u}+1}^{n} \Lambda^{(i)}(\bm{\theta}) \sigma(d\bm{\theta}) \right)^q.\end{equation} If the last expectation in \eqref{lemma3_upper} is bounded when $u\to0,$ then, one gets 
\begin{equation}\label{lemma_mom_est1}
E\mu_n^q(B_u(\mathbb{0}_d)) \leq C u^{dq} (E(\Lambda^{(0)}(\mathbb{0}_d))^q)^{n_u+1} \leq Cu^{dq-\log_bE(\Lambda^{(0)}(\mathbb{0}_d))^q}, \ {\rm when} \ u\to0.    
\end{equation} Since the estimate \eqref{lemma_mom_est1} is uniform in $n$, it holds true for $E\mu^q(B_{u}(\mathbb{0}_d))$ as well.

Let us show that the last expectation in \eqref{lemma3_upper} is bounded when $u\to0$. By the change of variables $\theta_1' = \theta_1 u$, one~gets
\[ u^{-d} \int\limits_{B_u(\mathbb{0}_d)}\prod\limits_{i=n_u}^n\Lambda^{(i)}(\theta_1',\theta_2,...,\theta_d)\sigma(d\theta_1'd\theta_2\cdot...\cdot d\theta_d)\] 
\begin{equation*}=u^{-d+1} \int\limits_0^{2\pi}\int\limits_{[0,\pi]^{d-2}}\int\limits_0^{1}\sin^{d-1}(u\theta_1)\sin^{d-2}\theta_2\cdot...\cdot\sin\theta_{d-1}\prod\limits_{i=n_u}^n\Lambda^{(i)}(u\theta_1,\theta_2,...,\theta_d) \prod\limits_{j=1}^d d\theta_j.\end{equation*}
As $\sin(u\theta_1)\leq \frac{\pi}{2}u\sin\theta_1, \ \theta_1 \in[0,\frac{\pi}{2}],\ u\in[0,1],$ one gets that the above is bounded by
\begin{equation*}
C\int\limits_0^{2\pi}\int\limits_{[0,\pi]^{d-2}}\int\limits_0^{1}\sin^{d-1}\theta_1\sin^{d-2}\theta_2\cdot...\cdot\sin\theta_{d-1}\prod\limits_{i=n_u}^n\Lambda^{(i)}(u\theta_1,\theta_2,...,\theta_d)\prod\limits_{j=1}^d d\theta_j\end{equation*}
\begin{equation}\label{lemma3:lambdatilde}=C\int\limits_{B_1(\mathbb{0}_d)} \prod\limits_{i=0}^{n-n_u} \widetilde{\Lambda}^{(i)}(\bm{\theta})\sigma(d\bm{\theta}):=C\widetilde{\mu}_{n-n_u}(B_1(\mathbb{0}_d)),\end{equation} where $\widetilde{\Lambda}^{(i)}(\bm{\theta}): = \Lambda^{(i+n_u)}(u\theta_1,\theta_2,...,\theta_d),\ \bm{\theta} =(\theta_1,\theta_2,..., \theta_d) \in\mathbb{S}^d, \ i\in\mathbb{N}_0.$ 

Let us show that the random fields $\widetilde{\Lambda}^{(i)}(\cdot), \ i\in\mathbb{N}_0,$ satisfy conditions of Theorem \ref{th2_mult}. Indeed, as random fields $\Lambda^{(i)}(\cdot), \ i\in\mathbb{N}_0,$ satisfy Assumption \ref{assump2_mult} and conditions of Theorem \ref{th2_mult}, by Lemma  \ref{lemma_asymp} and as the function $\rho(\cdot, \textit{\textbf{p}}),$ $\textit{\textbf{p}} = (p_1,p_2,...,p_k),\ p_j\geq0, \ j=\overline{1,k}, \ k\geq2, \ \sum p_j = p,$ is non-increasing, it~holds 
\[ E\left(\prod_{j=1}^k (\widetilde\Lambda^{(i)}(\bm{\theta}_j))^{p_j}\right) =  E\left(\prod_{j=1}^k (\Lambda^{(i+n_u)}(u\theta_1^{(j)},\theta_2^{(j)},...,\theta_d^{(j)}))^{p_j}\right)\]
\[ \leq \rho(b^{i+n_u}\min\limits_{m\neq h}d_{\mathbb{S}^d}((u\theta_1^{(m)},\theta_2^{(m)},...,\theta_d^{(m)}),(u\theta_1^{(h)},\theta_2^{(h)},...,\theta_d^{(h)})),\textit{\textbf{p}}) \]\[\leq \rho(Cb^{i}\min\limits_{m\neq h}d_{\mathbb{S}^d}(\bm{\theta}_m,\bm{\theta}_h),\textit{\textbf{p}}),\] for all $i\in\mathbb{N}_0,$ $\bm{\theta}_j \in B_1(\mathbb{0}_d),\ j = \overline{1,k},$ and some $C>0.$ Thus, random fields $\widetilde{\Lambda}^{(i)}(\cdot), \ i\in\mathbb{N}_0,$ satisfy Assumption \ref{assump2_mult} and conditions of Theorem \ref{th2_mult} such that the moments $E\widetilde{\mu}_{n-n_u}^p(B_1(\mathbb{0}_d)),$ $n\geq n_u,$ are uniformly bounded and the last expectation in \eqref{lemma3_upper} is bounded when $u\to0.$

Now, let us estimate $E\mu^q(B_u(\mathbb{0}_d))$ from above for $q\in(0,1).$ Application of the reverse H\"{o}lder's inequality $(EX^q)^{1/q} (EY^p)^{1/p}\leq E(XY), \ \frac{1}{q}+\frac{1}{p}=1, \ q\in(0,1), \ p<0,$ results in
\[ EX^q \leq \left( \frac{E(XY)}{(EY^p)^{1/p}}\right)^q = (E(XY))^{q} (EY^p)^{1-q}.\] By setting $X=\mu_{n}(B_u(\mathbb{0}_d))$ and $Y=(\mu_{n-1}(B_u(\mathbb{0}_d))(\Lambda^{(n)}(\mathbb{0}_d)))^{q-1},$ one obtains
\begin{equation}\label{lemma3:q01_upper} E\mu_{n}^q(B_u(\mathbb{0}_d)) \leq \left( E\left( \frac{\mu_{n}(B_u(\mathbb{0}_d))(\Lambda^{(n)}(\mathbb{0}_d))^{q-1}}{\mu_{n-1}^{1-q}(B_u(\mathbb{0}_d))} \right) \right)^q\left( E\left(\mu_{n-1}(B_u(\mathbb{0}_d)) (\Lambda^{(n)}(\mathbb{0}_d))\right)^{(q-1)p} \right)^{1-q}.\end{equation} For the first expectation in \eqref{lemma3:q01_upper}, it holds
\[E\left(  \frac{\mu_{n}(B_u(\mathbb{0}_d))(\Lambda^{(n)}(\mathbb{0}_d))^{q-1}}{\mu_{n-1}^{1-q}(B_u(\mathbb{0}_d))} \right) = E\left( \frac{\displaystyle\int_{B_u(\mathbb{0}_d)} \Lambda_{n-1}(\bm{\theta})\Lambda^{(n)}(\bm{\theta})(\Lambda^{(n)}(\mathbb{0}_d))^{q-1}\sigma(d\bm{\theta})}{\mu_{n-1}^{1-q}(B_u(\mathbb{0}_d))} \right)\]
\[ \leq E  \mu^q_{n-1}(B_u(\mathbb{0}_d)) \max_{\bm{\theta} \in B_u(\mathbb{0}_d) } E(\Lambda^{(n)}(\bm{\theta})(\Lambda^{(n)}(\mathbb{0}_d))^{q-1}) \leq E \mu^q_{n-1}(B_u(\mathbb{0}_d)) \rho(b^{n}u,\textit{\textbf{q}})\] \[\leq E  \mu^q_{n-1}(B_u(\mathbb{0}_d)) \rho(b^{n-n_t },\textit{\textbf{q}}),\] where the last estimate holds as $\rho(\cdot,\textit{\textbf{q}})$ does not decrease for $\textit{\textbf{q}}=(q-1,1), \ q\in(0,1).$ 

Thus, from \eqref{lemma3:q01_upper}, the above estimate and as $q=(q-1)p,$ it follows
\[ E\mu_{n}^q(B_u(\mathbb{0}_d)) \leq (E\mu^q_{n-1}(B_u(\mathbb{0}_d)) \rho( b^{n-n_t },\textit{\textbf{q}}))^{q} \left( E\left(\mu_{n-1}(B_u(\mathbb{0}_d)) (\Lambda^{(n)}(\mathbb{0}_d))\right)^{q} \right)^{1-q} \]
\[ = E\mu_{n-1}^q(B_u(\mathbb{0}_d)) E(\Lambda^{(0)} (\mathbb{0}_d))^q\left( \frac{\rho(b^{n-n_u},\textit{\textbf{q}})}{E(\Lambda^{(n)} (\mathbb{0}_d))^q} \right)^q,\] where we applied Assumption \ref{assump_isot} in the above expression. The recursive application of the above estimate results in
\begin{equation}\label{lemma3:est_below_q_small} E\mu_{n_u-1}^q(B_u(\mathbb{0}_d)) \leq  E\mu_{1}^q(B_u(\mathbb{0}_d)) (E(\Lambda^{(0)}(\mathbb{0}_d))^q)^{n_u-2}\prod\limits_{i=2}^{n_u-1} \left( \frac{\rho(b^{i-n_u},\textit{\textbf{q}})}{E(\Lambda^{(i-1)}(\mathbb{0}_d))^q} \right)^q.\end{equation}  As for 
$q\in(0,1)$ it holds $E\mu_{1}^q(B_u(\mathbb{0}_d))\leq(E\mu_{1}(B_u(\mathbb{0}_d))^q\sim Cu^{dq},\ u\to0,$ and as

\[(E(\Lambda^{(0)}(\mathbb{0}_d))^q)^{n_u} \leq C(E(\Lambda^{(0)}(\mathbb{0}_d))^q)^{-\log_bu} = Cu^{-\log_bE(\Lambda^{(0)}(\mathbb{0}_d))^q},\] it follows from \eqref{lemma3:est_below_q_small} 
\[ E\mu^q_{n_u-1}(B_u(\mathbb{0}_d)) \leq C u^{dq-\log_bE\left(\Lambda^{(0)}(\mathbb{0}_d)\right)^q} \prod\limits_{i=1}^{n_u-2} \left( \frac{\rho(b^{-i},\textit{\textbf{q}})}{E\left(\Lambda^{(i)} (\mathbb{0}_d)\right)^q} \right)^q,\ {\rm when} \ u\to0.\] By \eqref{lemma:mom_estim_cond} the above product is bounded. Thus, one gets

\[ E\mu_{n_u-1}^q(B_u(\mathbb{0}_d)) \leq C u^{dq-\log_bE\left(\Lambda^{(0)}(\mathbb{0})\right)^q}.\] As for $q\in(0,1)$ the sequence of random variables $\{\mu_n^q(B_u(\mathbb{0}_d)), \  n \in \mathbb{N}_0\}$ is a supermartingale, it holds $E\mu^q(B_u(\mathbb{0}_d)) \leq E\mu_{n_u-1}^q(B_u(\mathbb{0}_d)),$ from which follows the required estimate.

Now, let us estimate $E\mu^q(B_u(\mathbb{0}_d))$ from below for $q\in(0,1).$ By applying the  H\"{o}lder's inequality \eqref{holder_upper} with $\widetilde{p}=q,$ one obtains
\[\mu_{n}^q(B_u(\mathbb{0}_d)) = \left( \int\limits_{B_u(\mathbb{0}_d)} \left( \prod_{i=0}^{n_{\textit{\textbf{u}}}-1} \Lambda^{(i)}(\bm{\theta}) \left( \prod_{i=n_{\textit{\textbf{u}}}}^{n} \Lambda^{(i)}(\bm{\theta} )\right)^{1/q}\right)\left( \prod_{i=n_{\textit{\textbf{u}}}}^{n} \Lambda^{(i)}(\bm{\theta} )\right)^{1/p}\sigma(d\bm{\theta}) \right)^q \]
\[\geq \left( \int\limits_{B_u(\mathbb{0}_d)} \left( \prod_{i=0}^{n_{\textit{\textbf{u}}}-1} (\Lambda^{(i)}(\bm{\theta}))\right)^q \prod_{i = n_{\textit{\textbf{u}}}}^{n}\Lambda^{(i)}(\bm{\theta})  \sigma(d\bm{\theta})\right) \left(\int\limits_{B_u(\mathbb{0}_d)} \prod_{i = n_{\mathfrak{u}}}^{n}\Lambda^{(i)}(\bm{\theta} )  \sigma(d\bm{\theta})\right)^{q/p}.\]
Thus, by Assumption \ref{assump_isot}, one gets \begin{equation}\label{lemma_mom_est_below} E\mu_{n}^q(B_u(\mathbb{0}_d)) \geq  (E(\Lambda^{(0)}(\mathbb{0}_d))^q)^{n_u} u^{dq} E\left( u^{-d} \int\limits_{B_u(\mathbb{0}_d)}\prod\limits_{i=n_u}^n\Lambda^{(i)}(\bm{\theta})\sigma(d\bm{\theta})\right)^q. \end{equation} 
If the last expectation in \eqref{lemma_mom_est_below} is bounded from below by a positive constant when $u\to0,$ then, it holds 
\begin{equation}\label{lemma3:last_est} E \mu_{n}^q(B_u(\mathbb{0}_d)) \geq  C(E(\Lambda^{(0)}(\mathbb{0}_d))^q)^{n_u} u^{dq} \geq C u^{dq-\log_bE\left(\Lambda^{(0)}(\mathbb{0}_d)\right)^q}.\end{equation} Since the estimate \eqref{lemma3:last_est} is uniform in $n$, it holds true for $E\mu^q(B_{u}(\mathbb{0}_d))$ as well.

Let us consider the last expectation in \eqref{lemma_mom_est_below}. By following the steps leading to \eqref{lemma3:lambdatilde} and applying $\sin(u\theta_1)\geq \frac{2}{\pi}u\sin\theta_1,\ \theta_1\in[0,1], \ u\in[0,1],$ one obtains
\begin{equation}\label{lemma3:last_measure} u^{-d} \int\limits_{B_u(\mathbb{0}_d)}\prod\limits_{i=n_u}^n\Lambda^{(i)}(\bm{\theta})\sigma(d\bm{\theta}) \geq C\int\limits_{B_1(\mathbb{0}_d)} \prod\limits_{i=0}^{n-n_u} \widetilde{\Lambda}^{(i)}(\bm{\theta})\sigma(d\bm{\theta}):=C\widetilde{\mu}_{n-n_u}(B_1(\mathbb{0}_d)).\end{equation} Let us show that the random fields $\widetilde{\Lambda}^{(i)}(\cdot), \ i\in\mathbb{N}_0,$ satisfy conditions of Theorem \ref{th3_mult}. As the random fields $\Lambda^{(i)}, \ i=0,1,...,$ satisfy Assumption~\ref{assump3_mult}, it holds
\[ \min\limits_{\substack{\bm{\theta}_j\in B_x(\mathbb{0}_d) \\ j=\overline{1,k}}}E\bigg(\prod_{j=1}^k (\widetilde{\Lambda}^{(i)}(\bm{\theta}_j))^{p_j}\bigg) = \min\limits_{\substack{\bm{\theta}_j\in B_{ux}(\mathbb{0}_d) \\ j=\overline{1,k}}}E\bigg(\prod_{j=1}^k ({\Lambda}^{(i+n_u)}(\bm{\theta}_j))^{p_j}\bigg) \] \[ \geq \upupsilon(b^{i+n_u}ux, \textbf{\textit{p}}) \geq \upupsilon( b^{i}x, \textbf{\textit{p}}),\] for all $i\in\mathbb{N}_0, \ x\in[0,\pi].$ Thus, the random fields $\widetilde{\Lambda}^{(i)}(\cdot), \ i\in\mathbb{N}_0,$ satisfy Assumption \ref{assump3_mult}, and conditions of Theorem \ref{th3_mult} such that the limiting measure $\widetilde{\mu}(\cdot)$ is non-degenerate, i.e. $E\widetilde{\mu}^q(B_1(\mathbb{0}_d))>0.$ As for $q\in(0,1)$ the sequence of random variables $\{\widetilde{\mu}_n^q(B_1(\mathbb{0}_d)), \  n \in \mathbb{N}_0\}$ is a supermartingale, it holds $E\widetilde{\mu}^q(B_1(\mathbb{0}_d)) \leq E\widetilde{\mu}_{n-n_u}^q(B_1(\mathbb{0}_d)), n\geq n_u,$ such that by \eqref{lemma3:last_measure} the last expectation in \eqref{lemma_mom_est_below} is bounded from below by a positive constant when $u\to0.$

Finally, the lower bound for $q\geq1$ follows from Theorem~\ref{th3_mult} and Assumption \ref{assump_isot}. \end{proof}

The previous Lemma implies the following result.

\begin{theorem}\label{th_renyi}
Let the conditions of Lemma {\rm \ref{lemma:mom_estim}} hold true for each  $q\in[0,p],\ p\geq2.$  Then, the R\'enyi function of the limiting measure $\mu(\cdot)$ is
\[ \tau_{\mu}(q) = q-1-\frac{1}{d}\log_bE\left(\Lambda^{(0)}(\mathbb{0}_d)\right)^q, \ q\in[0,p].\]
\end{theorem}

\begin{proof}  By \cite[Theorem 2.1]{kaenmaki2012existence} for every set $\Delta_l^{(m)} \in\Delta_l,$ where $\Delta_l$ denotes the mesh formed by $l$-level dyadic decomposition of $\mathbb{S}^d,$ there exists a point $\bm{\theta}_l^{(m)}\in\Delta_l^{(m)},$ and the universal positive constants $c,C$ such that it holds $B_{c2^{-l}}(\bm{\theta}_l^{(m)})\subset \Delta_l^{(m)} \subset B_{C2^{-l}}(\bm{\theta}_l^{(m)}).$  By  the monotonicity of the limiting measure $\mu(\cdot),$ it holds $\mu(B_{c2^{-l}}(\bm{\theta_l^{(m)}}))\leq\mu(\Delta_l^{(m)})\leq\mu(B_{C2^{-l}}(\bm{\theta_l^{(m)}})), \ \Delta_l^{(m)}\in\Delta_l.$ 

Let us first estimate $\tau_{\mu}(\cdot)$ from above. By applying $\mu(\Delta_l^{(m)})\leq\mu(B_{C2^{-l}}(\bm{\theta_l^{(m)}})), \ \Delta_l^{(m)}\in\Delta_l,$ and the upper bound in \eqref{ineq:mom}, one gets
\[ \tau_{\mu}(q) = \liminf\limits_{l\to\infty} \frac{\log\left(\sum\limits_{\Delta_l^{(m)}\in \Delta_l} E \left( \mu(\Delta_l^{(m)})\right)^q\right)}{\log |\Delta_l^{(
0)}|} \leq  \liminf\limits_{l\to\infty} \frac{\log \sum\limits_{\Delta_l^{(m)}\in \Delta_l} C 2^{-l\left(dq-\log_bE\left(\Lambda^{(0)}(\mathbb{0}_d)\right)^q\right)}}{\log 2^{-ld}}\] 
\[\leq \liminf\limits_{l\to\infty}\frac{\log \left(C2^{ld-l\left(dq-\log_bE\left(\Lambda^{(0)}(\mathbb{0}_d)\right)^q\right)}\right)}{\log 2^{-ld}}=q-1-\frac{1}{d}\log_bE\left(\Lambda^{(0)}(\mathbb{0}_d)\right)^q.\] 

As $B_{c2^{-l}}(\bm{\theta}_l^{(m)})\subset \Delta_l^{(m)},$ one obtains the estimate from below analogously. \end{proof}

The next result provides the R\'enyi function in the geometric Gaussian scenario.
\begin{corollary}\label{cor:2} Let the conditions of Corollary {\rm \ref{cor_1}} hold true, i.e. $\mu(B)\in L_p, \ p\geq2,$ for $B\in \mathfrak{B},$ and $\mu(\cdot)$ is non-degenerate. Then, the limiting measure $\mu(\cdot)$ possesses the R\'enyi function
\[ \tau_{\mu}(q) = -\frac{ s^2}{2d\ln b}q^2 +\left(\frac{ s^2}{2d \ln b}+1\right)q-1, \ q\in[0,p].\]
\end{corollary}

\begin{remark}
The above result generalises Theorem $6.1$ in \cite{leonenko2021analysis}, where the authors consider geometric Gaussian scenario with exponentially decaying covariances on the sphere $\mathbb{S}^2$  and calculate the R\'enyi function in the interval~$[1,2].$
\end{remark}

\begin{proof} Let us show that the conditions of Lemma \ref{lemma:mom_estim} are satisfied.  As $d_{\mathbb{S}^d}(\bm{\theta}_m,\bm{\theta}_h)\geq d_{\rm Euc}(\bm{\theta}_m,\bm{\theta}_h), $ $\bm{\theta}_m,\bm{\theta}_h\in\mathbb{S}^d,$ and $r(\cdot)$ is non-increasing, for $\bm{q} = (q-1,1), \ q\in(0,1),$ following the steps leading to \eqref{rho_gaus}, one gets
\[ \rho(b^{-i},\textbf{\textit{q}}) = \exp\left(\frac{1}{2}\left(s^2((q-1)^2-(q-1))+2(q-1)r\left( b^{-i} \right)    \right) \right).\] As $E\left(\Lambda^{(i)}(\mathbb{0}_d)\right)^q = \exp\left( \frac{s^2q(q-1)}{2}\right),$ $i\in\mathbb{N}_0,$ one obtains 
\[ \sum\limits_{i=0}^\infty \ln \left( \frac{\rho(b^{-i}, \textbf{\textit{q}})}{E\left(\Lambda^{(i)}(\mathbb{0}_d)\right)^q} \right) = (1-q)\sum\limits_{i=0}^\infty \left(s^2 - r\left( b^{-i} \right)\right).\] Thus, as $r(\cdot)$ is non-increasing, the condition \eqref{lemma:mom_estim_cond} is satisfied if \eqref{cond_gau_4} holds. 

For $q\geq1$ let us choose a vector $\widetilde{\textbf{\textit{p}}}=(\widetilde{p},\widetilde{p}), \ \widetilde{p} = q/2.$ Then, by \eqref{upepsilon} it holds that
\[ \upupsilon(b^{-i},\widetilde{\textbf{\textit{p}}}) = \exp\left( \frac{1}{2}\left(s^2\left(\frac{q^2}{2} - q\right) + \frac{q^2}{2}r(2b^{-i}) \right) \right)\] and 
\[ \sum\limits_{i=0}^\infty \ln\left( \frac{ E\left(\Lambda^{({i})}(\mathbb{0}_d)\right)^q}{\upupsilon(b^{-i}, \widetilde{\textbf{\textit{p}}})} \right) =  \frac{q^2}{4}\sum\limits_{i=0}^\infty(s^2 - r(2b^{-i})),\] which means that the condition \eqref{cond:nondeg} of Therem \ref{th3_mult} is satisfied if \eqref{cond_gau_4} holds true. Thus, the conditions of Lemma~\ref{lemma:mom_estim} are satisfied and it holds
\[ \tau_{\mu}(q) = q-1-\frac{1}{d}\log_bE\left(\Lambda^{(0)}(\mathbb{0}_d)\right)^q = -\frac{ s^2}{2d \ln b}q^2 +\left(\frac{ s^2}{2d \ln b}+1\right)q-1,  \ q\in[0,p].\]\end{proof}

\section*{Funding} This work was supported by the Australian Research Council's Discovery Projects funding scheme under Grant DP220101680 and by Ripple Impact Fund 2022-247584 (5855).

\section*{Disclosure statement} No potential conflict of interest was reported by the author(s).

\bibliographystyle{abbrv}
\bibliography{mybib}

\section*{Appendix}

\textit{Proof of Lemma {\rm \ref{lemma_asymp}}} By \eqref{dist} and as $\sin(\varphi)\geq \frac{2}{\pi}\varphi, \ \varphi\in[0,\frac{\pi}{2}],$ it holds
\[ d_{\mathbb{S}_d}^2((u\theta_1^{(1)},\theta_2^{(1)},...,\theta_d^{(1)}), (u\theta_1^{(2)},\theta_2^{(2)},...,\theta_d^{(2)})) \]
\begin{equation}\label{eq:lemma_2_eq} \leq \frac{\pi^2}{4}d_{{\rm Euc}}^2((u\theta_1^{(1)},\theta_2^{(1)},...,\theta_d^{(1)}), (u\theta_1^{(2)},\theta_2^{(2)},...,\theta_d^{(2)})).\end{equation} Using \cite[Lemma 2.1]{cheng2016excursion}, one obtains
\[ d_{{\rm Euc}}^2((u\theta_1^{(1)},\theta_2^{(1)},...,\theta_d^{(1)}), (u\theta_1^{(2)},\theta_2^{(2)},...,\theta_d^{(2)}))\] \[ = \left(4\sin^2\left(\frac{u(\theta_{1}^{(1)}-\theta_{1}^{(2)})}{2}\right) +2\sin(u\theta_{1}^{(1)})\sin(u\theta_{1}^{(2)}) ( 1 - \cos( \theta_2^{(1)} - \theta_2^{(2)} )) \right.\] \begin{equation}\label{lemma2:expans}\left.+...+2\sin(u\theta_{1}^{(1)})\sin(u\theta_{1}^{(2)}) \prod_{i=2}^{d-1}\sin(\theta_{i}^{(1)})\sin(\theta_{i}^{(2)}) ( 1 - \cos( \theta_d^{(1)} - \theta_d^{(2)} ) ) \right).\end{equation} By applying $\sin\left( u \varphi \right) \leq \frac{\pi}{2} u \sin (\varphi), \ \varphi\in[0,\frac{\pi}{2}],\ u \geq 0,$ in the above expression, one gets 
\[ d_{\rm Euc}^2((u\theta_1^{(1)},\theta_2^{(1)},...,\theta_d^{(1)}), (u\theta_1^{(2)},\theta_2^{(2)},...,\theta_d^{(2)})) \leq \frac{\pi^2 u^2}{4}d_{\rm Euc}^2(\bm{\theta}_1, \bm{\theta}_2).\] As by \eqref{dist} $d_{\rm Euc}(\bm{\theta}_1,\bm{\theta}_2) \leq d_{\mathbb{S}^d}(\bm{\theta}_1,\bm{\theta}_2), \ \bm{\theta}_1,\bm{\theta}_2\in\mathbb{S}^d,$
from \eqref{eq:lemma_2_eq} and the above one obtains 
\[ d_{\mathbb{S}_d}^2((u\theta_1^{(1)},\theta_2^{(1)},...,\theta_d^{(1)}), (u\theta_1^{(2)},\theta_2^{(2)},...,\theta_d^{(2)})) \leq \frac{\pi^4 u^2}{16} d_{\mathbb{S}_d}^2(\bm{\theta}_1, \bm{\theta}_2).\] The lower bound is obtained analogously using
\[d_{\mathbb{S}_d}((u\theta_1^{(1)},\theta_2^{(1)},...,\theta_d^{(1)}), (u\theta_1^{(2)},\theta_2^{(2)},...,\theta_d^{(2)})) \]\[ \geq d_{\rm Euc}((u\theta_1^{(1)},\theta_2^{(1)},...,\theta_d^{(1)}), (u\theta_1^{(2)},\theta_2^{(2)},...,\theta_d^{(2)})),\] applying \eqref{lemma2:expans} and successively using inequalities $\sin(u\varphi)\geq \frac{2}{\pi}u\sin(\varphi), \ \varphi \in[0,\frac{\pi}{2}],\ u\geq0,$ and $d_{\rm Euc}(\bm{\theta}_1,\bm{\theta}_2) \geq \frac{2}{\pi}d_{\mathbb{S}^d}(\bm{\theta}_1,\bm{\theta}_2),$ $\bm{\theta}_1,\bm{\theta}_2\in\mathbb{S}^d.$  \hfill{} \qedsymbol

\end{document}